\documentclass{elsarticle}

\usepackage{amsthm}
\usepackage{amsmath}
\usepackage{amssymb}
\usepackage{graphicx}
\usepackage{amscd}

\theoremstyle{plain}
\newtheorem{theorem}{Theorem}[section]
\newtheorem{lemma}[theorem]{Lemma}
\newtheorem{proposition}[theorem]{Proposition}
\newtheorem{corollary}[theorem]{Corollary}

\theoremstyle{definition}
\newtheorem{definition}{Definition}[section]

\theoremstyle{remark}
\newtheorem{remark}{Remark}

\begin{document}

\title{Singularities of Functions on the Martinet Plane, Constrained Hamiltonian Systems and Singular Lagrangians}
\author{Konstantinos Kourliouros}
\address{Imperial College London, Department of Mathematics,  Huxley Building 180 Queen's Gate,\\
South Kensington Campus, London SW7, United Kingdom}
\ead{k.kourliouros10@imperial.ac.uk}

\begin{abstract}
We consider here the analytic classification of pairs $(\omega,f)$ where $\omega$ is a germ of a 2-form on the plane and $f$ is a quasihomogeneous function germ with isolated singularities. We consider only the case where $\omega$ is singular, i.e. it vanishes non-degenerately  along a smooth line $H(\omega)$ (Martinet case) and the function $f$ is such that the pair $(f,H(\omega))$ defines an isolated boundary singularity. In analogy with the ordinary case (for symplectic forms on the plane) we show that the moduli in the classification problem are analytic functions of 1-variable and that their number is exactly equal to the Milnor number of the corresponding boundary singularity. Moreover we derive a normal form for the pair $(\omega,f)$ involving exactly these functional invariants. Finally we give an application of the results in the theory of constrained Hamiltonian systems, related to the motion of charged particles in the quantisation limit in an electromagnetic field, which in turn leads to a list of normal forms of generic singular Lagrangians (of first order in the velocities) on the plane.
\end{abstract}


\maketitle


\section{Introduction-Main Results}

In several local analysis problems arising in mathematical physics, control theory, dynamical systems e.t.c. one is led to consider the classification problem for pairs $(\omega,f)$, where $\omega$ is a germ of a closed 2-form on a manifold $M$ and $f$ is a function germ, with or without singularities. The most studied case is when the 2-form is nondegenerate, i.e. it defines a symplectic structure on $M$. Then $f$ can be viewed as a Hamiltonian function and the classification problem reduces to the well known problem of symplectic classification of singularities of functions (c.f. \cite{RA}, \cite{A'}, \cite{Bi}). In this direction, the 2-dimensional problem is drastically different than the higher dimensional one. Indeed, a symplectic structure on the plane is an area (volume) form and the classification of Hamiltonian functions on the plane becomes a problem of volume-preserving geometry. There, analytic normal forms exist in abundance (c.f. \cite{F3}, \cite{F1}, \cite{V} as well as \cite{C1}, \cite{G1}, \cite{G2}, \cite{Va} for related results) in contrast to the symplectic category \cite{Si}.

In this paper we consider a generalisation of the classification problem for pairs $(\omega,f)$ on a 2-manifold $M$, where now the 2-form $\omega$ is allowed to have singularities and thus it does not define a symplectic structure everywhere on $M$. This situation is typical when we consider Hamiltonian systems with constraints (c.f. \cite{D}, \cite{F}, \cite{J}, \cite{K}, \cite{Lic},  \cite{M1}, \cite{M2}, \cite{M3}, \cite{P1}, \cite{P2}).  

In analogy with the unconstrained case, we may define a Constrained Hamiltonian System (CHS) on a 2-manifold $M$, simply as a pair $(\omega,f)$ consisting of a function $f$ and a 2-form $\omega$ on $M$ as above. One may obtain such a system as the restriction of a Hamiltonian system in a general ambient symplectic manifold, on a 2-dimensional surface representing the constraints. Let $X_f$ be the ``Hamiltonian vector field'' associated to the pair $(\omega,f)$ through the equation:
\[X_f\lrcorner \omega=df.\]
This vector field is in general not defined and smooth everywhere on $M$. The obstruction to the existence and/or uniqueness of $X_f$ is obviously the set of zeros $H(\omega)$ of the 2-form $\omega$. In the theory of singularities of Constraints Systems (c.f. \cite{SZ}, \cite{Z2}) it is called the \textit{Impasse Hypersurface}, while in the theory of differential systems is usually called the \textit{Martinet hypersurface} (c.f. \cite{M2}), in honor of J. Martinet who was the first who studied systematically singularities of differential forms \cite{Mar}. The problem is thus to classify CHS  at impasse points (away from the impasse points the problem reduces to the ordinary symplectic classification of functions). 

It is easy to see that the germ of a generic singular 2-form $\omega$ on the plane can be reduced to Martinet normal form \cite{Mar}:
\[\omega=xdx\wedge dy.\]
The geometric invariants of the 2-form $\omega$ on the plane are just its Martinet curve of zeros $H(\omega)=\{x=0\}$, along with an orientation (in the real analytic (smooth) case) induced by the two symplectic structures in its complement. The orientation of the Martinet curve plays no role in the initial definition of singularity classes for the pair $(f,\omega)$ and thus, all the singularities are defined by the relative positions of the germ $f$ with the Martinet curve. In particular, the pair $(H(\omega),f)$ can be viewed as defining a germ of a ``boundary singularity'' at the origin of the plane i.e. such that:
\begin{itemize}
\item either $f$ has an isolated critical point at the origin,
\item or $f$ is non-singular but its restriction $f|_{H(\omega)}$ on the Martinet curve has an isolated critical point at the origin.
\end{itemize}
Thus, in order to study the singularities of pairs $(\omega,f)$, one may fix an arbitrary boundary singularity $(f,H)$ and study possible normal forms of degenerate 2-forms $\omega$, whose zero set $H(\omega)$ is exactly the curve $H$, under the action of the pseudogroup $\mathcal{R}_{f,H}$ of diffeomorphisms preserving the boundary singularity.  

The study of isolated boundary singularities has been initiated by V. I. Arnol'd in \cite{A1} (see also \cite{A0}, \cite{A00}, \cite{A2} for general references) where he extended the $\mathbf{A}$, $\mathbf{D}$, $\mathbf{E}$ classification of simple singularities to include also the $\mathbf{B}$, $\mathbf{C}$, $\mathbf{F}_4$ series of Weyl groups in the scheme of singularity theory. The list of simple normal forms obtained by Arnol'd is given for convenience in Table \ref{tab:1}.
\begin{table}
\caption{Simple singularities of functions on a 2-manifold with boundary $H=\{x=0\}$}
\label{tab:1}       
\begin{tabular}{llll}
\hline\noalign{\smallskip}
 $\textbf{A}_{\mu}$ & $\textbf{B}_\mu$ & $\textbf{C}_\mu$ &  $\textbf{F}_4$\\
\noalign{\smallskip}\hline\noalign{\smallskip}
$x+y^{\mu+1}$ & $x^{\mu}+y^2$ & $xy+y^{\mu}$ & $x^2+y^3$ \\
\noalign{\smallskip}\hline\noalign{\smallskip}
$\mu \geq 1$ & $\mu \geq 2$ & $\mu \geq 2$ & $\mu=4$ \\
\noalign{\smallskip}\hline
\end{tabular}
\end{table}

The number $\mu$ in the list is an important invariant. It is called the Milnor number, or the multiplicity of the boundary singularity and it is intimately related with the classification problem of pairs $(\omega,f)$ we will study, as is the ordinary Milnor number of an isolated singularity in the symplectic case (c.f. \cite{F1} and also \cite{G2}). The manifestation of the Milnor number in symplectic (isochore in higher dimensions) classification problems in the ordinary (without boundary) case, is due to the following fact: the corresponding deformation space of a germ of a symplectic form $\omega$ on the plane (relative to diffeomorphisms preserving the germ $f$) is exactly equal to the Brieskorn module of the singularity $f$ and in particular, according to the Brieskorn-Deligne-Sebastiani theorem \cite{B}, \cite{S}, any such deformation space will be a free module of rank equal to the Milnor number of $f$, over the ring $\mathbb{C}\{f\}$ of analytic functions on (the values of) $f$. Thus, the classification problem reduces to a problem of relative de-Rham cohomology.    

The classification  problem in the Martinet case which we will study here, is not much different than the symplectic classification problem, despite the fact that the corresponding relative de-Rham cohomology theory for isolated boundary singularities has not yet been established in the literature. In particular one may conjecture, in analogy with the ordinary case, that  the number of functional invariants in the classification of pairs $(\omega,f)$ is exactly equal to the Milnor number of the boundary singularity $(f,H(\omega))$ and moreover, there exists a normal form for the pair involving exactly these invariants (which by the way will be analytic functions of 1-variable). Here we will prove this statement only for the special case where the pair $(f,H(\omega))$ is a quasihomogeneous boundary singularity (Theorem \ref{t2}) and in particular for any of the simple germs in Arnol'd's list in Table \ref{tab:1} (as well as for other unimodal, e.t.c. boundary singularities which are still quasihomogeneous c.f. \cite{A0}, \cite{A2} and \cite{A1} for the corresponding lists). 

Our proof for the quasihomogeneous case is very similar (almost identical in some parts) to the one proposed by J. P. Francoise in \cite{F1}, \cite{F2} for the ordinary case. While not that general, in order to include all isolated boundary singularities (i.e. away from Arnol'd's list), Francoise's method has the advantage that it is algorithmic in nature: one may construct step by step the characteristic invariants of the pair $(\omega,f)$ as well as the bounds of the corresponding norms. Also it has minimum prerequisites, such as the relative Poincar\'e lemma and the relative de-Rham's division lemma (see Section 2), while a general proof should go through several sheaf theoretic techniques. Thus, a main part of the paper is to prove (Theorem \ref{t1}) that the corresponding deformation module of a Martinet 2-form $\omega$ (relative to diffeomorphisms preserving a simple boundary singularity $(f,H(\omega))$) is a free $\mathbb{C}\{f\}$-module of rank equal to the Milnor number of the boundary singularity, i.e. equal to the corresponding number $\mu$ in Table \ref{tab:1}.            

One final remark: the classification of functions $f$ and Martinet 2-forms $\omega$ in higher dimensions is a much more complicated problem and the results of this paper do not generalise in this case. Probably, formal normal forms do exist but we don't discuss this here. Instead, we give an application of the 2-dimensional results in a problem arising in the geometric theory of Hamiltonian systems with constraints, that is, the problem of classification of generic singular Lagrangians (of first order in the velocities) on the plane, under variational (gauge) equivalence (Theorem \ref{tl}). Such Lagrangians, already studied by Dirac \cite{D} for the purposes of quantisation and which appear in high energy physics (c.f. \cite{F}, \cite{J}, \cite{P2}), in hydrodynamics and general vortex theory (c.f. \cite{A'}, \cite{Ko} and references therein) and also in control theory and sub-riemannian geometry \cite{M1}, \cite{M2}, \cite{M3} and several other instances, give rise to Euler-Lagrange equations which define a constrained Hamiltonian system $(\omega,f)$ and thus, they are subjectable to the analysis in this paper.

\section{Deformations of Singular Symplectic Strucures and Boundary Singularities}

We fix a coordinate system $(x,y)$ at the origin of $\mathbb{C}^2$ such that the boundary is given by the equation $H=\{x=0\}$. To a boundary singularity $(f,H)$  we associate its local algebra (c.f. \cite{A1})
\[\mathcal{Q}_{f,H}=\mathcal{O}/(x\frac{\partial f}{\partial x},\frac{\partial f}{\partial y}),\]
where the ideal $J_{f,H}$ in the denominator is the tangent space to the $\mathcal{R}_H$-orbit of $f$, i.e. under diffeomorphisms (right-equivalences) preserving the boundary $H$ (as usual $\mathcal{O}$ is the algebra of germs of analytic functions at the origin). We call it, in analogy with the ordinary case the \textit{Jacobian ideal of the boundary singularity $(f,H)$}. Its codimension, i.e. the $\mathbb{C}$-dimension $\mu$ of the vector space $\mathcal{Q}_{f,H}$ is called the \textit{multiplicity} or \textit{Milnor number of the boundary singularity} $(f,H)$ and it is an important invariant: it is related to the Milnor number $\mu_1$ of $f$:
\[\mu_1=dim_{\mathbb{C}}\mathcal{O}/(\frac{\partial f}{\partial x},\frac{\partial f}{\partial y})\]
and the Milnor number $\mu_0$ of its restriction $f|_H$ on the boundary:
\[\mu_0=dim_{\mathbb{C}}\mathcal{O}|_{H}/(\frac{\partial f}{\partial y}|_{H}),\]
by the formula:
\[\mu=\mu_1+\mu_0.\]

Topologically it can be interpreted as the rank of the relative homology group $H_1(X_s,X_s\cap H;\mathbb{Z})$ of the pair $(X_s,X_s\cap H)$, $s\neq 0$, of smooth Milnor fibers of $f$ and $f|_{H}$, as one may easily deduce from the exact homology sequence induced by the inclusion $X_s\cap H\subset X_s$ (the fiber $X_s\cap H$ consists of a finite number of points). In particular, according to a theorem of Arnol'd \cite{A1} which generalises Milnor's theorem \cite{Mil} for the boundary case, the space $X_s/X_s\cap H$ has the homotopy type of a bouquet of $\mu$ circles.  It follows that the pair $(X_s,X_s\cap H)$ is a Riemann surface with Betti number $\mu_1$ and $\mu_0+1$ distinguished points. For $s\rightarrow 0$, the $\mu_1$ cycles of $X_s$ and the $\mu_0$ segments joining the distinguished points shrink at the origin: they are called vanishing cycles and half-cycles respectively and they form a basis of $H_1(X_s,X_s\cap H;\mathbb{Z})$ (see \cite{A1}).  

Denote now by $\Omega^{\cdot}$, the complex of germs of analytic differential forms at the origin (where $\Omega^0=\mathcal{O}$ the algebra of analytic functions) and by $x\Omega^{\cdot}$ the subcomplex of forms that ``vanish on $H=\{x=0\}$'', in the sense that their coefficients belong in the ideal $x\Omega^0(=(x)\subset \Omega^0)$ generated by the equation of $H$. Notice that any form vanishing on $H$ vanishes automatically when evaluated at tangent vectors of $H$ (i.e. it has zero pull back by the embedding $H\hookrightarrow (\mathbb{C}^2,0)$), but the converse does not hold (take for example the 1-form $dx$). We distinguish by writing $\Omega^i_{H}$ for the space of $i$-forms with zero pull-back on $H$. Notice that with this notation $\Omega^0_H=x\Omega^0$,  $\Omega^2_H=\Omega^2$ (identically) and $x\Omega^1_H= x\Omega^0dx+x^2\Omega^0dy$. The space $x\Omega^2$ may be identified with the space of 2-forms whose zero set contains the curve $H=\{x=0\}$.  We write $x\Omega^2_*$ for the space of Martinet 2-forms (with zero set exactly equal to $H$). We will need the following local version of a type of relative Poincar\'e lemma for the complex $x\Omega^{\cdot}_{H}$ (c.f. \cite{DJZ}, \cite{Gi}):

\begin{lemma}
\label{l1}
For any closed $i$-form $\alpha \in x\Omega^i_{H}$ there exists an $(i-1)$-form $\beta \in x\Omega^{i-1}_H$ such that $\alpha=d\beta$.  
\end{lemma}       
\begin{proof}
By the classical Poincar\'e lemma the 1-parameter family of maps $F_t(x,y)=(tx,ty)$, $t\in [0,1]$, is a contraction at the origin, it preserves $H$, $F_t(H)\subset  H$ and is such that: $F^*_1\alpha=\alpha$, $F^*_0\alpha=0$ and $F^*_1\alpha=d\beta$, where $\beta$ is defined by
\[\beta=\int_0^1F_t^*(V_t\lrcorner \alpha)dt\]
and the vector field $V_t=dF_t/dt=(x,y)$ is defined as the generator of $F_t$. Notice that by definition $V_t$ is tangent to $H$ for all $t$. Now, since $\alpha$ vanishes on $H$ to second order, the $(i-1)$-form $V_t\lrcorner \alpha$ vanishes also on $H$ and since $V_t$ is tangent to $H$ it follows that $V_t\lrcorner \alpha \in x\Omega^{i-1}_H$. By the fact that $F_t(H)\subset H$, it follows that $\beta \in x\Omega^{i-1}_H$.   
\end{proof}

Fix now a pair $(f,H)$, where $f$ has an isolated singular point at the origin of finite multiplicity $\mu=\mu_1+\mu_0$. The differential $df$ defines an ideal in the algebra $\Omega^{\cdot}$ of germs of differential forms at the origin, which induces an ideal in all the subalgebras $x\Omega^{\cdot}$, and $x\Omega^{\cdot}_H$ we consider. The lemma below gives necessary and sufficient conditions for the ideal membership problem. It is an analog in the relative case of de-Rham's division lemma \cite{Der} (see also \cite{E}). 

\begin{lemma}
\label{l2}
\noindent \begin{itemize}
\item[(a)] Let $\omega \in x\Omega^2$. Then $\omega=df\wedge \eta$ holds for some $\eta \in x\Omega^1_H$ if and only if $\omega \in J_{f,H}x\Omega^2$. 
\item[(b)] For any 1-form $\alpha \in x\Omega^1_H$ such that $\alpha \wedge df=0$, there exists a function germ $g\in x\Omega^0_H$ such that $\alpha=gdf$. 
\end{itemize}
\end{lemma}  

\begin{proof}
\noindent \begin{itemize}
\item[(a)] the proof is an obvious calculation
\item[(b)] It suffices to show that the relative de-Rham's division lemma is true for the complex $\Omega^{1}_H$: indeed, if we write $\alpha=x\alpha_1$ for some $\alpha_1 \in \Omega^1_H$, then $df\wedge (x\alpha_1)=xdf\wedge \alpha_1=0$ and since $x$ is a non-zero divisor it will follow from the relative de-Rham division lemma in $\Omega^{1}_H$ that there exists a function $g_1\in \Omega^0_H$ such that $\alpha_1=g_1df$. But then $\alpha=xg_1df$ and the germ $g=xg_1 \in x\Omega^0_H$ as we wanted. To show now that the relative division lemma is true in $\Omega^1_H$ it suffices to notice that the condition $df\wedge \alpha=0$ for any 1-form $\alpha \in \Omega^1_H$ forces it again to be of the form $\alpha=x\alpha_0$ for some $1$-form $\alpha_0$. Indeed, this follows from restriction $(df\wedge \alpha)|_{x=0}=0$ and the fact that $df|_{x=0}\neq 0$ (because $f$ has isolated singularities). Then, as before, the equation $df\wedge (x\alpha_0)=xdf\wedge \alpha_0=0$ implies that $df\wedge \alpha_0=0$ and thus, by the ordinary de Rham division lemma, there exists a function germ $g_0$ such that $\alpha_0=g_0df$. Thus $\alpha=xg_0df$ and the lemma is proved. 
\end{itemize} 
\end{proof}

\begin{remark}
It is important to notice here that the condition $df\wedge \alpha=0$ implied on a 1-form $\alpha$ vanishing on $H$ (i.e. for $\alpha \in \Omega^1_H$ or $\alpha \in x\Omega^1_H$) prevents it from being closed (=exact). Indeed, $df\wedge dh=0$ means that the function $h$ is constant on the fibers of $f$ and thus it is a holomorphic function on the values of $f$: $h=h(f)$. But then, since $h$ should vanish on $x=0$ it means that it should vanish everywhere in the neighborhood under consideration (because the fibers of $f$ are almost everywhere transversal to $x=0$). 
\end{remark}

By the extension of Tougeron's theorem on the finite determinacy of boundary singularities $(f,H)$ proved by V. I. Matov \cite{Mat} we may suppose that $H=\{x=0\}$ and $f$ is polynomial of sufficiently high degree ($\geq \mu+1$). Write $\mathcal{R}_{f,H}$ for the pseudogroup of symmetries of the pair $(f,H)$. In the following lemma we identify the set of (infinitesimal) trivial deformations of Martinet 2-forms relative to $\mathcal{R}_{f,H}$-action:
\begin{lemma}
\label{l3}
Let $\omega$ be a germ of a Martinet 2-form. The tangent space to the orbit of $\omega$ under the $\mathcal{R}_{f,H}$-action consists of all 2-forms of the form:
\[\frak{r}_{f,H}(\omega)=df\wedge d(x\Omega^0_H).\] 
\end{lemma} 
\begin{proof}
Let $v$ be an element of the Lie algebra $\frak{r}_{f,H}$. The infinitesimal deformation of $\omega$ associated to $v$ is by definition an element of the form $L_v\omega$ (where $L$ is the Lie derivative). We have that $df\wedge(v\lrcorner \omega)=L_v(f)\omega=0$ and thus, by de-Rham's division with $df$, there exists a function germ $g$ such that $v\lrcorner \omega=gdf$. Since the 1-form $v\lrcorner \omega$ vanishes on $H$ to second order (because both $v$ and $\omega$ vanish on $H$) we conclude that $g\in x\Omega^0_H$. It follows   that $d(v\lrcorner \omega)=L_v\omega=df\wedge d(-g)$, $g\in x\Omega^0_H$. Conversely, let $g\in x\Omega^0_H$ be such that there exists a vector field $v$ with $L_v\omega=df\wedge dg$. In fact define $v$ as the dual of the 1-form $gdf$ through $\omega$: $v\lrcorner \omega=-gdf$ (this is possible because $g$ vanishes on $H$). Obviously $L_v(f)=0$ and $v$ is vanishes on $H$ since $g$ vanishes on $H$ to second order. Thus $v\in \frak{r}_{f,H}$ and the lemma is proved. 
\end{proof}

It follows that the quotient space
\[\mathcal{D}_{f,H}(\omega)=x\Omega^2/df\wedge d(x\Omega^0_H)\]
consists of the nontrivial infinitesimal deformations of the Martinet 2-form $\omega$ relative to the symmetries of the boundary singularity. Along with the $\mathbb{C}$-linear space structure, the deformation space $\mathcal{D}_{f,H}(\omega)$ (which we will denote simply by $\mathcal{D}(\omega)$ ) has a natural $\mathbb{C}\{f\}$-module structure with multiplication by $f$.  We call it the \textit{deformation module of the Martinet germ $\omega$}. In the next section we will show that this module is a free module of rank $\mu=\mu_1+\mu_0$ over $\mathbb{C}\{f\}$. This statement is analogous to the classical Brieskorn-Deligne-Sebastiani\footnote{M. Sebastiani proved the freeness of the Brieskorn module in higher dimensions.} theorem \cite{B}, \cite{S} for ordinary singularities (i.e. without boundary). This finiteness result along with the following proposition are cornerstones in the classification problem.   
\begin{proposition}
\label{p1}
Fix a boundary singularity $(f,H)$. Let $\omega$ and $\omega'$ be two germs of Martinet 2-forms at the origin, such that $\omega-\omega' \in df\wedge d(x\Omega^0_{H})$. Then there exists a diffeomorphism $\Phi \in \mathcal{R}_{f,H}$ such that $\Phi^*\omega'=\omega$.
\end{proposition}   
\begin{proof}
The proof is by the homotopy method. Consider a 1-parameter family of Martinet 2-forms connecting $\omega$ and $\omega'$:
\[\omega_t=\omega+tdf\wedge dh, \quad h\in x\Omega^0_{H},\]
so that $\omega_0=\omega$ and $\omega_1=\omega'$.
We seek a 1-parameter family of vector fields $v_t\in \frak{r}_{f,H}$ such that
\[L_{v_t}\omega_t=0\Leftrightarrow d(v_t\lrcorner \omega_t)=df\wedge d(-h),\]
for all $t\in [0,1]$. Choose $v_t$ by $v_t\lrcorner \omega_t=hdf$. It preserves $\omega_t$ and it is also in $\frak{r}_{f,H}$ by the same reasoning as in the previous lemma. It follows that the time 1-map of the flow $\Phi_t$ of $v_t$ sends $\omega_0$ to $\omega_1$.
\end{proof}

\section{Finiteness and Freeness of the Deformation Module}

The classical Brieskorn-Deligne-Sebastiani theorem for an isolated singularity $f$ \cite{B}, \cite{S} and also \cite{Mal2}, states that the quotient space $\Omega^2/df\wedge d\Omega^0$ is a free $\mathbb{C}\{f\}$-module of rank equal to the Milnor number of the singularity $f$. It is proved using sheaf theoretic techniques and the properties of the Gauss-Manin connection on the vanishing cohomology bundle associated to the singularity. For some special cases of singularities though, other simpler proofs exist (c.f. \cite{V}, \cite{V1} for the Morse case). For the class of all quasihomogeneous singularities a proof was given by J. P. Fran\c{c}oise in \cite{F2} (see also \cite{F1}). Fran\c{c}oise's proof, while not covering the whole class of isolated singularities, it has the advantage that it is algorithmic, i.e. one may construct term by term the coefficients in the expansion of a form in the Brieskorn module as well as their bounds. Fran\c{c}oise's algorithm is rather general and it holds also in the global, polynomial case (c.f. \cite{Y}). Here we will provide the necessary modifications to include the boundary case as well. In particular we will prove:
\begin{theorem}
\label{t1}
Let $(f,H)$ be a quasihomogeneous boundary singularity at the origin of Milnor number $\mu$ and let $\omega$ be a germ of a Martinet 2-form whose zero set is exactly the curve $H$. Then the deformation module of $\omega$ is a free module of rank $\mu$ over $\mathbb{C}\{f\}$:
\[\mathcal{D}(\omega)\cong \mathbb{C}\{f\}^{\mu}.\]
\end{theorem}

To prove the theorem, suppose that $f$ is a quasihomogeneous polynomial of type $(m_1,m_2;1)$, $m_i\in \mathbb{Q}_+$, i.e. such that $f(t^{m_1}x,t^{m_2}y)=tf(x,y)$. Denote by 
\[E_f=m_1x\frac{\partial}{\partial x}+m_2y\frac{\partial }{\partial y}\]
the Euler vector field of $f$, i.e. such that $E_f(f)=f$. Write also $M=m_1+m_2$. Then the following division lemma holds:

\begin{lemma}
\label{divlem}
If $(f,H)$ is a quasihomogeneous boundary singularity at the origin of the plane, then the following identity holds:
\begin{equation}
\label{df}
df\wedge x\Omega^1_H=f(x\Omega^2)+df\wedge d(x\Omega^0_H).
\end{equation}
\end{lemma}   
\begin{proof} 
It suffices to find, for a given 1-form vanishing on the boundary $\eta \in x\Omega^1_H$, a 2-form $\theta \in x\Omega^2$ and a function $h \in x\Omega^0_H$ vanishing on the boundary to second order, such that
\[df\wedge \eta=f\theta +df\wedge dh.\]
But $f\theta=df\wedge(E_f\lrcorner \theta)$ and so the equality above reduces to 
\[df\wedge (\eta-E_f\lrcorner \theta -dh)=0,\]
i.e. to 
\[E_f\lrcorner \theta=\eta-dh.\]
Taking exterior differential, it suffices to find $\theta$ such that:
\[L_{E_f}\theta=d\eta.\]
We view $L_{E_f}$ as an operator in formal series:
\[L_{E_f}:\hat{x\Omega}^2\rightarrow \hat{x\Omega}^2, \hspace{0.2cm} L_{E_f}=m_1x\frac{\partial }{\partial x}+m_2y\frac{\partial }{\partial y}+M.\]
This is obviously an invertible operator since for any monomial $x^iy^j$, $b=(i,j)$, $i>0$ we have:
\[L_{E_f}x^iy^j=(<m,b>+M)x^iy^j,\]
where $<m,b>+M$ never vanishes. Thus given $\theta$ we can find a formal solution $\eta \in \hat{x\Omega}^1_H$. This solution can easily be extended to an analytic solution $\eta \in x\Omega^1_H$ in a fundamental system of neighborhoods of the origin (see Lemma \ref{rdl} below). The lemma is proved. 
\end{proof}

Now we describe the analog of Fran\c{c}oise's algorithm for the generation of the coefficients in the decomposition of $\omega$ in the deformation module.

\subsection{Construction of a Formal Basis of the Deformation Module} 

First we construct a formal basis of the $\mathbb{C}[[f]]$-module $\Hat{\mathcal{D}}(\omega)$ (i.e. of the formal deformation module): 

Choose a monomial basis $\{e_i(x,y)\}_{i=1}^{\mu}$ of the local algebra $\mathcal{Q}_{f,H}$ of the boundary singularity and lift it to a basis of monomial 2-forms $\omega_i=\{e_i(x,y)dx\wedge dy\}_{i=1}^{\mu}$ in $\Omega^2_{f,H}=\Omega^2/df\wedge \Omega^1_H$. Any 2-form $\omega \in x\Omega^2$ can be written as $\omega=x\tilde{\omega}$ for some 2-form $\tilde{\omega}$. Decompose now $\tilde{\omega} \in \hat{\Omega}^2$ in $\hat{\Omega}^2_{f,H}$:
\[\tilde{\omega}=\sum_{i=1}^{\mu}c_i\omega_i+df\wedge \tilde{\eta},\]
where $\tilde{\eta} \in \hat{\Omega}^1_H$ is a 1-form vanishing on $H$ (and defined uniquely by $\omega$ modulo terms of the form $gdf$) and $c_i \in \mathbb{C}$ for $i=1,...,\mu$. This decomposition induces also a decomposition, after multiplication with the function $x$, in the space $\hat{x\Omega}^2$, in the sense that:
\begin{equation}
\label{de11}
\omega=x\sum_{i=1}^{\mu}c_i\omega_i+df\wedge x\tilde{\eta}.
\end{equation}
Write now $\eta=x\tilde{\eta} \in \hat{x\Omega}^1_H$ and decompose the 2-form $df\wedge \eta$ according to the division Lemma \ref{divlem} and plug it to equation (\ref{de11}) above:
\begin{equation}
\label{de22}
\omega=x\sum_{i=1}^{\mu}c_i\omega_i+f\theta +df\wedge dh,
\end{equation}
where $\theta=x\tilde{\theta} \in \hat{x\Omega}^2$ and $h\in \hat{x\Omega}^0_H$. Continuing that way, decompose $\tilde{\theta}$ in $\hat{\Omega}^{2}_{f,H}$:
\[\tilde{\theta}=\sum_{i=1}^{\mu}c_i^1\omega_i+df\wedge \tilde{\eta_1},\]
where again $\tilde{\eta}_1 \in \hat{\Omega}^1_H$ (is defined by $\tilde{\theta}$ modulo terms $gdf$) and the $c_i^1 \in \mathbb{C}$ are constants. Multiplying by $x$ and plugging this back to equation (\ref{de22}) we obtain the new decomposition:
\begin{equation}
\label{de33}
\omega=x\sum_{i=1}^{\mu}(c_i+c_i^1f)\omega_i+fdf\wedge \eta_1+df\wedge dh,
\end{equation}
where $\eta_1=x\tilde{\eta_1} \in \hat{x\Omega}^1_H$. Now use again the division Lemma \ref{divlem} for the 2-form $df\wedge \eta_1$ to obtain new $\theta_1=x\tilde{\theta_1}\in \hat{x\Omega}^2$, $h_1\in \hat{x\Omega}^0_H$ such that:
\[df\wedge \eta_1=f\theta_1+df\wedge dh_1\]
 and plug it back to (\ref{de33}) above to get:
\[\omega=x\sum_{i=1}^{\mu}(c_i+c_i^1f)\omega_i+f^2\theta_1+df\wedge d(fh_1+h).\]
Continuing that way with the 2-form $\theta_1$, e.t.c. we obtain at the $p$-th iterate a decomposition of the form:
\[\omega=x\sum_{i=1}^{\mu}(c_i+c_i^1f+...+c_i^pf^p)\omega_i+df\wedge d(f^ph_p+...+fh_1+h)+o(f^{p+1}).\]
The term $o(f^{p+1})$ belongs, for $p\rightarrow \infty$, to the intersection of all maximal ideals $\cap_p\frak{m}^p$, i.e. it goes to zero in the $\mathbb{C}[[f]]$-module $\hat{\mathcal{D}(\omega)}$. Thus, the algorithm converges in the Krull topology.

\subsection{Proof of Convergence in the Analytic Category}

For convenience in notation, we change coordinates from $(x,y)$ to $x=(x_1,x_2)$ (so that $H=\{x_1=0\}$). We write in this notation $x^b=x_1^ix_2^j$ for a monomial $x^iy^j$ and a vector $b=(i,j)$ in $\mathbb{N}^2$.  Let $r=(r_1,r_2) \in \mathbb{R}^2_+$ and let 
\[D(r)=\{(x_1,x_2)\in \mathbb{C}^2/|x_1|\leq r_1, \hspace{0.2cm} |x_2|\leq r_2\}\]
be a polycylinder in $\mathbb{C}^2$. We consider the pseudo-norm $|.|_r$ in $\mathcal{O}$ defined by:
\[|\phi|_r=\sum_b |\phi_b|r^b,\]
where $\phi=\sum_{b}\phi_bx^b$ is an analytic function. We denote by $\mathcal{O}_r$ the subset of $\mathcal{O}$ for which the pseudo-norm $|.|_r$ is finite and thus defines a norm.  
For $\phi=(\phi_1,...,\phi_k)\in \mathcal{O}_r^k$ we have accordingly:
\[|\phi|_r=\sum_{i=1}^k|\phi_i|_r.\]

Identify now $\Omega^2_r$ with $\mathcal{O}_r$ and $\Omega^1_r$ with $\mathcal{O}_r\times \mathcal{O}_r$. Obviously $(\Omega^1_{H})_r$ can be identified with the subspace $\mathcal{O}_r\times (x\mathcal{O})_r$ and $(x\Omega^1_H)_r$ with $(x\mathcal{O})_r\times (x^2\mathcal{O})_r$.  The map $u=df\wedge :x\Omega^1_{H}\rightarrow x\Omega^2$ is a $\mathcal{O}$-linear map and induces a map:
\[u:(x\Omega^1_H)_r\rightarrow (x\Omega^2)_r.\]
A section of $u$ is a $\mathbb{C}$-linear map:
\[\lambda : (x\Omega^2)_r\rightarrow (x\Omega^1_{H})_r,\]
such that  $u=u\lambda u$. We can see from the definition that $\lambda=\frac{.}{df}$ is division with $df$. 

We say that a section $\lambda$ is adapted to the polydisc $D(r)$ (or to $r$) if $\lambda$ is a continuous mapping between Banach spaces, i.e. there exists a constant $C_r$ such that:
\[|\lambda (\theta)|_r\leq C_r|\theta|_r,\]
for all $\theta \in (x\Omega^2)_r$.  A consequence of Malgrange's priviledged neighborhoods theorem is the following:
\begin{proposition}[\cite{Mal1}]
Given $u$ there is a section $\lambda$ such that the set of polydiscs $D(r)$ onto which is adapted, forms a fundamental system of neighborhoods of the origin.
\end{proposition}

Now we continue the proof by making precise the choise of the section. We will use the following two lemmata, whose proofs are exactly the same as in \cite{F2}. The first concerns the bounds obtained by division with $df$ and the second the corresponding bounds obtained by the relative Poincar\'e lemma.

\begin{lemma}
\label{rdl}
If $\theta \in (x\Omega^2)_r$ is such that $L_{E_{f}}\theta=d\eta$, then:
\[|\theta|_r\leq \frac{1}{m_0r_0}|\eta|_r,\]
where $m_0=min(m_1,m_2)$ and $r_0=min(r_1,r_2)$.
\end{lemma}
\begin{proof}
It is exactly the same as in \cite{F2}, Lemma 3.1.2. For completeness, write $\eta=\eta_2dx_1+\eta_1dx_2$ with $\eta_i=\sum_b\eta^i_bx^b$, $i=1,2$, where, since $\eta$ vanishes on $H=\{x_1=0\}$, the vector $b$ in $\eta_1$ is of the form $b=(b_1,b_2)$, $b_1\geq 1$. Then by direct computation: 
\[\theta=\sum_{i=1}^2\sum_b\frac{b_i}{<m,b>+M-m_i}\eta^i_bx^{b-I_i},\]
where $I_1=(1,0)$, $I_2=(0,1)$ are unit vectors in $\mathbb{N}^2$. From this the lemma follows. 
\end{proof}
\begin{lemma}
\label{rpl}
For any closed 1-form $\pi \in (x\Omega^1_H)_r$, there exists a function $\zeta \in (x\Omega^{0}_{H})_r$ such that: 
\[\pi=d\zeta, \hspace{0.2cm} |\zeta|_r\leq R|\pi|_r,\]
where $R=r_1+r_2$.
\end{lemma}
\begin{proof}
It follows by the the proof of the relative Poincar\'e Lemma \ref{l1}.  Indeed, the function $\zeta$ is defined by $\zeta=\int_0^1F_t^*(E\lrcorner \pi)dt$ where $F_t(x)=(tx_1,tx_2)$ is a contraction mapping at the origin and $E=dF_t/dt$ is the Euler vector field: $E=x_1\partial /\partial x_1+x_2\partial /\partial x_2$. The bound on the norm follows then by a direct computation. 
\end{proof}

Now, the finiteness of the deformation module in the analytic category may be stated in the following form:

\begin{proposition}
Let $D(r)$ be a polycylinder onto which the section $\lambda$ is adapted. Then there exists a smaller polycylinder $D(r')\subset D(r)$ such that for any $\omega \in (x\Omega^2)_r$ there exist an analytic function $\xi \in(x\Omega^0_H)_{r'}$ and $\mu$ analytic functions $c_i(f)\in \Omega^0_r$ such that:
\[\omega=x\sum_{i=1}^{\mu}c_i(f)\omega_i+df\wedge d\xi,\]
with the following explicit bounds:
\[|h|_{r'}\leq \frac{RC_r(1+\frac{MR}{m_0r_0})}{1-|f|_{r'}\frac{C_r}{m_0r_0}}|\omega|_r,\]
\[|c_i(f)|_{r'}\leq \frac{|\omega|_r}{1-|f|_{r'}\frac{C_r}{m_0r_0}}.\]
\end{proposition}
\begin{proof}
The proof is again the same as in the ordinary case \cite{F2} (Theorem 3.2). We present it for completeness. Start with the first decomposition
\[\omega=\sum_{i=1}^{\mu}c_i\omega_i+df\wedge \eta,\]
with the bound $|\eta|_r\leq C_r|\omega|_r$. Then we solve the equation $d\eta=L_{E_f}\theta$ and we find a $\theta$ such that, according to Lemma \ref{rdl}:
\[|\theta|_r\leq \frac{1}{m_0r_0}|\eta|_r\leq \frac{C_r}{m_0r_0}|\omega|_r.\]
Then by the relation $dh=E_f\lrcorner \theta -\eta$ and Lemma \ref{rpl} we obtain the bounds:
\[|h|_r\leq R(1+\frac{MR}{m_0r_0})C_r|\omega|_r.\]
Decomposing this way we obtain at the $p$-th iterate:
\[|\theta_p|_r\leq (\frac{C_r}{m_0r_0})^p|\omega|_r,\]
\[|h_p|_r\leq RC_r(1+\frac{MR}{m_0r_0})(\frac{C_r}{m_0r_0})^{p-1}|\omega|_r.\]
Choose now $r'$ such that $|f|_{r'}(C_r/m_0r_0)<1$. Then, for the term $df\wedge d(\sum_{i=0}^pf^{i}h_{i})$ (where $h_0=h$) we have the bounds:
\[|\sum_{i=0}^pf^ih_{i}|_{r'}\leq \sum_{i=0}^p|f|_{r'}^i|h_{i}|_{r'}\leq \sum_{i=0}^p|f|_{r'}^i|h_{i}|_r\]
and for the term $\sum_{i=1}^{\mu}(\sum_{j=0}^pc_i^jf^{j})\omega_i$ (where $c_i^0=c_i$) we have the corresponding bounds:
\[|\sum_{j=0}^{p}c_i^jf^j|_{r'}\leq \sum_{j=0}^p|c_i^j|_{r}|f|_{r'}.\]
From this the theorem follows.   
\end{proof}

\subsection{Proof of Freeness of the Deformation Module}

Again as in \cite{F2}, it suffices to show that the $\mathbb{C}\{f\}$-module $\mathcal{D}(\omega)$ is torsion free. Suppose then that there exists a function germ $h\in x\Omega^0_H$ such that $f\omega=df\wedge dh$ for some 2-form $\omega \in x\Omega^2$. We will need to show that there exists an analytic function $\xi \in x\Omega^0_H$ such that $\omega=df\wedge d\xi$. But by assumption, we have that $df\wedge (E_f\lrcorner \omega -dh)=0$ and thus, by the relative de-Rham division lemma, there exists a germ $g\in x\Omega^0_{H}$ such that $E_f\lrcorner \omega -dh=gdf$. Taking the exterior differential, this relation reads:
\[L_{E_f}\omega=df\wedge d(-g).\]
Take now quasihomogeneous decomposition of the form $\omega$, $\omega=\sum_k\omega_k$. For any $k$ we have
\[\omega_k=df\wedge \frac{-(dg)_{k-1+M}}{k+M},\]
and since the Lie derivative commutes with the differential, we obtain the existence of an analytic function:
\[\xi=\sum_k \frac{-g_{k-1+M}}{k+M},\]
such that $\omega=df\wedge d\xi$. Obviously $\xi \in x\Omega^0_H$ and freeness is proved.

\subsection{Choice of a Basis}

To construct a basis of $\mathcal{D}(\omega)$ we consider the $\mathbb{C}\{f\}$-module $F=df\wedge x\Omega^1_H/df\wedge d(x\Omega^0_H)$.  We have a natural inclusion of $\mathbb{C}\{f\}$-modules $F\subset \mathcal{D}(\omega)$. Multiplication by $f$ in $\mathcal{D}(\omega)$ gives obviously elements inside $F$ and so $f\mathcal{D}(\omega)\subset F$. On the other hand, by the quasihomogeneous division with $df$ (Lemma \ref{divlem}) we obtain that the class of any 2-form of the form $df\wedge \eta$, $\eta \in x\Omega^1_H$, can be represented by the class of a 2-form $f\theta$, $\theta \in x\Omega^2$. From this it follows that: 
\[f\mathcal{D}(\omega)=F.\]
Thus we obtain a sequence of isomorphisms of $\mathbb{C}$-vector spaces:
\[\mathcal{D}(\omega)/f\mathcal{D}(\omega)\cong \mathcal{D}(\omega)/F\cong x\Omega^2_{f,H} \cong x\mathcal{Q}_{f,H}\]
which is again a $\mu$-dimensional vector space. Thus, by Nakayama lemma a basis of monomials $e_i(x,y)$, $i=1,...,\mu$ of the local algebra $\mathcal{Q}_{f,H}$ lifts to a basis $xe_i(x,y)dx\wedge dy$ in the deformation module $\mathcal{D}(\omega)$.

\section{Local Normal Forms and Functional Invariants}

The following theorem is the analog of a theorem of Fran\c{c}oise \cite{F1} in the ordinary (symplectic, or volume-preserving) case. It concerns the local normal forms of Martinet 2-forms under the action of the boundary-singularity preserving diffeomorphism group $\mathcal{R}_{f,H}$.

\begin{theorem}
\label{t2}
Let $(f,H)$ be a quasi-homogeneous boundary singularity of finite multiplicity $\mu$. Then, for any germ of a Martinet 2-form $\omega$  at the origin, there exist $\mu$ analytic functions $c_i\in \mathbb{C}\{t\}$ and a diffeomorphism $\Phi \in \mathcal{R}_{f,H}$, such that $\omega$ is reduced to the normal form:
\[\Phi^*\omega=x\sum_{i=1}^\mu c_i(f)e_i(x,y)dx\wedge dy,\]
where the classes of the monomials $e_i(x,y)$ form a basis of the local algebra $\mathcal{Q}_{f,H}$. Moreover, the $\mu$ functions $c_i$ are uniquely determined by the pair $(f,\omega)$.
\end{theorem}
\begin{proof}
The existence of the normal form is obtained immediately by the homotopy method of Proposition \ref{p1} and the finiteness Theorem \ref{t1}. Thus, it suffices to prove only the uniqueness of the coefficients $c_i \in \mathbb{C}\{t\}$. For this, we will need the following lemma:
\begin{lemma}
\label{l0}
For any $\omega \in x\Omega^{2}$ and any $\Phi \in \mathcal{R}_{f,H}$, there exists an $h \in x\Omega^0_{H}$ such that:
\[\omega-\Phi^*\omega=df\wedge dh.\]
\end{lemma}
\begin{proof}[Proof of the Lemma.]
Briefly, interpolate $\Phi$ by a 1-parameter formal subgroup $\Phi_t \in \hat{\mathcal{R}}_{f,H}$, i.e. $\Phi_0=Id$, $\Phi_1=\Phi$ and 
\[\Phi^*_tf=f, \hspace{0.2cm} \Phi_t(H)=H.\] 
This is always possible due to general properties of the group $\mathcal{R}_{f,H}$ (the proof goes as in \cite{F} for the ordinary case). Then
\[\omega-\Phi^*\omega=\int_0^1\frac{d}{dt}\Phi^*_t\omega dt=\int_0^1\Phi_t^*(L_{\hat{X}}\omega) dt,\]
where $\hat{X}$ is the formal vector field generated by the 1-parameter subgroup $\Phi_t$. Since $\Phi_t$ preserves $(x=0,f)$ for all $t\in [0,1]$ we have that
\[L_{\hat{X}}\omega=d(\hat{X}\lrcorner \omega)=df\wedge d\hat{g},\]
for some formal function $\hat{g}\in \hat{\Omega}^0_{x=0}$ and in particular:
\[\omega-\Phi^*\omega=df\wedge d\int_0^1\Phi^*_tg dt=df\wedge d\hat{h}.\]
Now if we consider the decomposition of the $2$-form $\omega-\Phi^*\omega$ in the deformation module $\mathcal{D}(\omega)$:
\[\omega-\Phi^*\omega=\sum_{i=1}^{\mu}\psi_i(f)\omega_i+df\wedge dh,\]
then this it can be read as a decomposition in the formal module $\hat{\mathcal{D}}(\omega)$. Comparing these two decompositions we immediately obtain $\psi_i(f)=0$ for all $i=1,...,\mu$.
\end{proof}

\noindent \textit{(Continuation of the proof of Theorem \ref{t2}).} Decompose first $\omega$ in the deformation module $\mathcal{D}(\omega)$:
\[\omega=\sum_{i=1}^{\mu}\tilde{c}_i(f)\omega_i+df\wedge dh,\]
 and take the difference of $\omega$ with $\Phi^*\omega$:
\[\omega-\Phi^*\omega=\sum_{i=1}^{\mu}(\tilde{c}_i(f)-c_i(f))\omega_i+df\wedge dh.\]
Then from Lemma (\ref{l0}) above it immediately follows that $\tilde{c}_i(t)=c_i(t)$. 
\end{proof}

\subsection{Geometric Description of the Moduli for the Nondegenerate case}

The complete geometric description of the $\mu$ functional invariants $c_i(t)$ associated to the pair $(\omega,f)$ is a difficult task and it is a part of what is usually called Gauss-Manin theory.  Below we will consider only the nondegenerate case $\mu=\mu_0=1$, i.e. for the pair:
\begin{equation}
\label{s1}
\omega=xc(f)dx\wedge dy, \hspace{0.3cm} f(x,y)=x+y^2,
\end{equation}
where $c(0)\neq 0$.

Let us consider first a small ball at the origin of $\mathbb{C}^2$  such that the fibers of $f(x,y)=t$ are transversal to the boundary of this ball over the points $t$ of a sufficiently small disc in $\mathbb{C}$, centered at the origin (the critical value of the restriction $f|_H$). Modifying the neighborhoods under consideration sufficiently, we may suppose that the fibers of the restriction $f|_H$ on the Martinet curve (which consists of two points away from the origin, for $t\neq 0$) are also transversal to the restriction of the boundary of the initial ball on the Martinet curve $H$. The latter consists of two points and thus transversality with the fibers of $f|_H$ means simply that they do not meet on $H$, i.e. the fibers $f^{-1}(t)\cap H$ are bounded within a sufficiently small segment of the Martinet curve. The intersection of each of the fibers of $f$ with the interior of the chosen ball, is an open Riemann surface $X_t$ with a set of distinguished points $X_t\cap H$. Let $\gamma(t)$ be a 1-parameter family of relative cycles on the pair of fibers representing a relative homology class in $H_1(X_t,X_t\cap H; \mathbb{C})$, so that $\gamma(t)$ is obtained  by continuous deformation of some relative cycle $\gamma(t_0)$ over a smooth pair $(X_{t_0}, X_{t_0}\cap H)$.  As is easily seen, for $t$ real and positive, the pair of fibers is contractible to its real part and as $t\rightarrow 0$ the fiber $X_t\cap H$ shrinks to a point  (see figure \ref{fig:1}). Arnol'd called the relative cycle $\gamma(t)$ arising this way, vanishing half cycle \cite{A1}: 
\[\gamma(t)=\{(x,y)\in \mathbb{R}^2/x\geq 0,\hspace{0.2cm} x+y^2=t, \hspace{0.2cm} t< \epsilon\}.\]

Obviously, if $\omega$ is a germ of a Martinet 2-form and $\alpha$ is a primitive of $\omega$, then the integral 
\[V(t)=\int_{\gamma(t)}\alpha,\]
is an invariant of the pair $(\omega,f)$. In a realisation of the 2-form as a magnetic (curvature) 2-form, the integral $V(t)$ is nothing more that the magnetic flux on the 2-cell enclosed by the vanishing half-cycle.
It is easy to see that this integral is a holomorphic function of $t$ and thus we may consider its derivative $V'(t)$:
\[V'(t)=\int_{\gamma(t)}\frac{d\alpha}{df}=\int_{\gamma(t)}\frac{\omega}{df}.\]
The integrand of this integral, the so called Gelfand-Leray form \cite{A00}, \cite{Ph}, is defined as follows: let $\omega_0=xdx\wedge dy$ be the standard Martinet 2-form. Then if we denote by $E_f$ the Euler vector field of $f$ the following relation holds:
\[f\omega_0=df\wedge \alpha_0, \]
where $\alpha_0=E_f\lrcorner \omega_0=x^2dy-(xy/2)dx$ and of course $d\alpha_0=(5/2)\omega_0$.  Now, since $\omega=c(f)\omega_0$ we have that 
\[\frac{\omega}{df}=\frac{c(f)}{f}\alpha_0\]
and thus:
\[V'(t)=\frac{c(t)}{t}\int_{\gamma(t)}\alpha_0.\]
Now the latter integral $V_0(t)$ can be evaluated immediately, $V_0(t)=4/3t^{5/2}$ and it can be interpreted as the magnetic flux enclosed by the vanishing half-cycle $\gamma(t)$. Thus we have:
\begin{equation}
\label{pf}
tV'(t)=c(t)V_0(t).
\end{equation}
From this equation we obtain the expression for the invariant: 
\[c(t)=(3/4)t^{-3/2}V'(t).\] 
Its geometric explanation is direct: it measures the rate of change of the magnetic flux on the family of fibers, enclosed by the vanishing half-cycle $\gamma(t)$ for $t$ varying close to zero. 

\begin{figure}
\label{fig:1}
\begin{center}
\includegraphics[width=7cm, height=5cm]{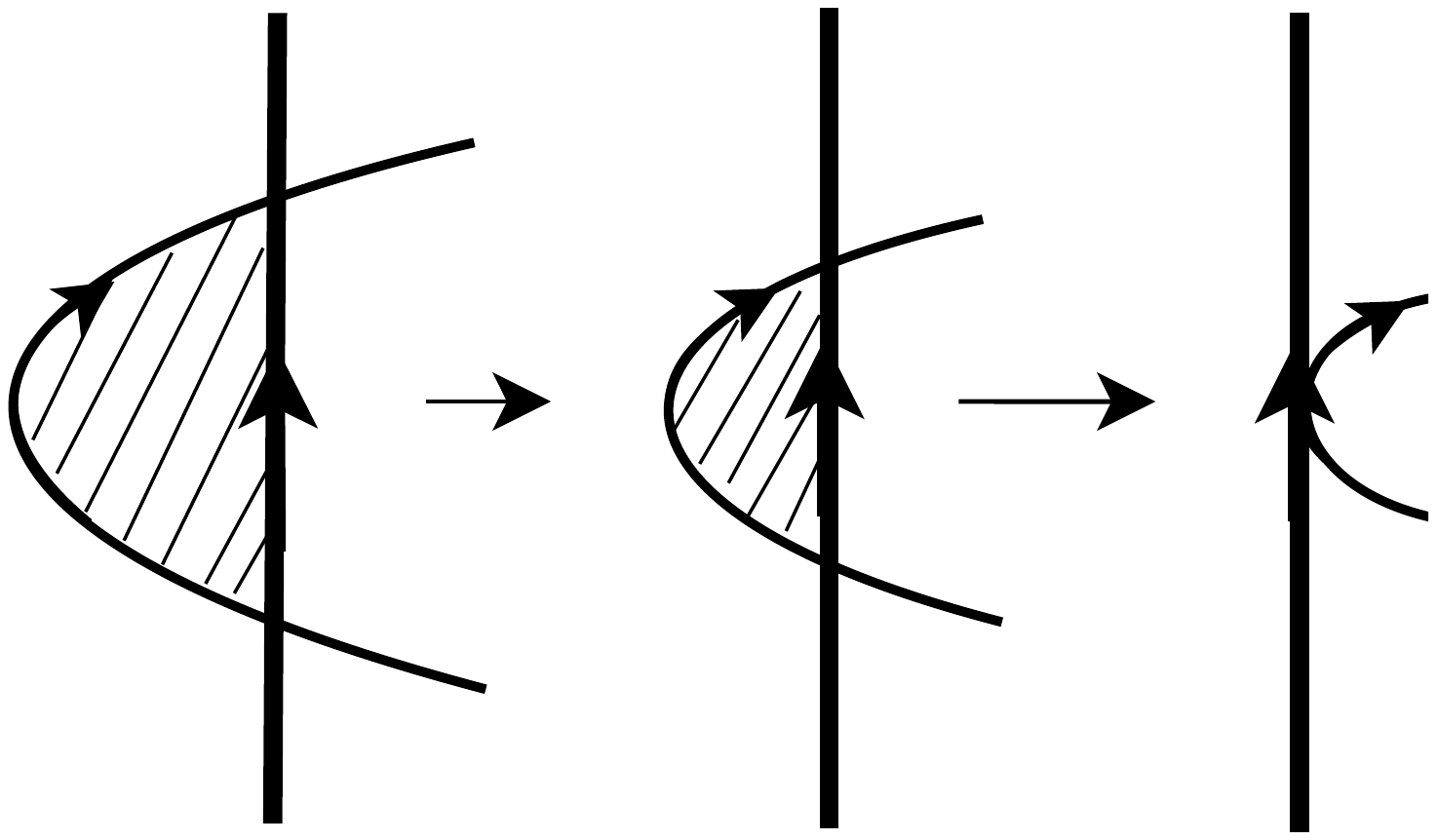}
\end{center}
\caption{Local model for a vanishing half-cycle of a boundary singularity. The area of the shaded region is $const. t^{3/2}$, while the magnetic flux is $const. t^{5/2}$.}
\end{figure}

\subsection{Martinet Normal Form for Nondegenerate Boundary Singularities}

The problem of classification of functions $f$ under the action of the group $\mathcal{R}_{\omega}$ of diffeomorphisms preserving the Martinet germ, is not so easy as is the corresponding classification of Martinet 2-forms $\omega$ relative to the $\mathcal{R}_{f,H}$-action (where as usual $H=H(\omega)$). To see this it suffices to check that the infinitesimal equation 
\[L_{v_t}f_t=\phi,\]
where $f_t=f+t\phi$ is a family of functions and $v_t$ a family of vector fields, does not admit any solutions inside the algebra $\mathfrak{r}_{\omega}$ of vector fields preserving the Martinet germ $\omega$.  Despite this fact, for the nondegenerate case $\mu=\mu_0=1$, a normal form involving exactly 1 functional invariant exists.

Fix the Martinet germ $\omega=xdx\wedge dy$. The following theorem describes the $\mathcal{R}_{\omega}$-orbit of the $\mathbf{A}_1$-boundary singularity $f=x+y^2$. It is the analog of Vey's isochore Morse lemma \cite{V} in the Martinet case:
\begin{corollary}
\label{t3}
Let $f:(\mathbb{C}^2,0)\rightarrow (\mathbb{C},0)$ be a function germ such that the origin is a regular point for $f$ but nondegenerate (Morse) critical point for the restriction $f|_{x=0}$ on the Martinet curve. Then there exists a  diffeomorphism $\Phi \in \mathcal{R}_{\omega}$  and a uniquely defined function $\psi \in \mathbb{C}\{t\}$, $\psi(0)=0$, $\psi'(0)=1$ such that
\begin{equation}
\label{e2}
\Phi^*f=\psi(x+y^2).
\end{equation} 
\end{corollary} 

\begin{proof}
By Theorem \ref{t2} above we may choose a coordinate system $(x,y)$ such that $(x=0, f=x+y^2)$ and $\omega=xc(f)dx\wedge dy$, where $c\in \mathbb{C}\{t\}$ is a function, nonvanishing at the origin. We may suppose that $c(0)=1$. We will show that there exists a change of coordinates $(x,y)\stackrel{\Phi}{\mapsto} (x',y')$ such that the pair $(x=0,f=x+y^2)$ goes to $(x=0, \psi(f))$ for some function $\psi$ and $\omega$ is reduced to Martinet normal form. To do this, we set $x'=xv(f)$, $y'=y\sqrt{v(f)}$, where $v\in \mathbb{C}\{t\}$ is some function with $v(0)=1$ (so $\Phi(x,y)=(x',y')$ is indeed a boundary-preserving diffeomorphism). With any such function $v$ we have $\Phi^*f=\psi(f)$ for the function $\psi(f)=fv(f)$, with $\psi(0)=0$ and $\psi'(0)=1$. Now it suffices to choose $v$ so that the map $\Phi$ satisfies $v(f)det\Phi_*=c(f)$, i.e. such that the following initial value problem is satisfied for the function $w=v^{5/2}$:
\begin{equation}
\label{pf1}
\frac{2}{5}tw'(t)+w(t)=c(t), \quad w(0)=1.
\end{equation}
As is easily verified this admits an analytic solution given by the formula
\[w(t)=t^{-\frac{5}{2}}\int_0^t\frac{5}{2}s^{\frac{3}{2}}c(s)ds.\]
\end{proof}

\section{An Application: Motions of Generalised Particles in the Quantisation Limit and Local Normal Forms of Singular Lagrangians}

We consider an example from Lagrangian mechanics which concerns the motion of a charged particle on a Riemann surface in the strong coupling (quantisation) limit with an electromagnetic field, or more generaly with an Abelian gauge field (\cite{F}, \cite{J}, \cite{P2}). There are several approaches and reformulations of the problem, the most appropriate for our case being that of a generalised particle (c.f.  \cite{A'}, \cite{Bi}), a variant of which we present below.

We fix a 2-dimensional riemannian manifold $M$. We consider the family of mechanical systems described by a regular Lagrangian function $L:TM\rightarrow \mathbb{R}$  ``quadratic in the velocities'', i.e. such that in any local trivilisation of the tangent bundle with coordinates $(x,\dot{x})$ it can be expressed as:
\begin{equation}
\label{gl}
L=mL_2+eL_1+\nu L_0,
\end{equation} 
where the functions $L_i=L_i(x,\dot{x})$ are homogeneous in the velocities $\dot{x}$ of degree $i$ and:
\begin{itemize}
\item[-] $L_2(x,.)=\sum g_{ij}(x)dx_idx_j$ is a nondegenerate quadratic form $g$ on $M$ (the riemannian metric) representing the kinetic energy of the system
\item[-] $L_1(x,.)=\sum \alpha_i(x)dx_i$  corresponds to a 1-form (vector potential) $\alpha$ of gyroscopic forces (such as magnetic forces e.t.c.) represented by the 2-form $\omega=da$,
\item[-] $L_0(x,.)=-f(x)$ is independent in the velocities and represents the scalar potential of other external forces acting upon the system (such as electric e.t.c.),
\item[-] $m\in \mathbb{R}$, $e\in \mathbb{R}$ and $\nu \in \mathbb{R}$ are the coupling constants ($m$ is the mass, $e$ is the charge e.t.c.), which may be viewed as formal parameters.
\end{itemize}
It is known that motions of the generalised particle emanating from a pont $x_0\in M$ are smooth curves $t\mapsto x(t)$ on $M$, $x(0)=x_0$, which satisfy the generalised Euler-Lagrange equations:
\begin{equation}
\label{el}
m\frac{D\dot{x}}{dt}=e\dot{x}\lrcorner da-\nu df,
\end{equation}   
where $D/dt$ is the covariant derivative associated to the riemannian metric $g$ and $\lrcorner$ is the interior multiplication of a vector field with the 2-form $\omega=da$. The right hand-side of equation (\ref{el}) is known in the theory of electromagnetism as the Lorentz force. In particular, the motions $x:[t_0,t_1]\rightarrow M$ of the particle are exactly the critical points of the action functional:
\begin{equation}
\label{f}
\mathcal{A}=\int_{t_0}^{t_1}L(x(t),\dot{x}(t))dt.
\end{equation}

For the case $\nu=0$, i.e. in the absence of external forces, the geometry of this variational problem with Lagrangian $L=mL_2+eL_1$, has been studied extensively in terms of Subriemannian geometry and control theory (c.f. \cite{M1}, \cite{M2}, \cite{M3}).  There, the eventual singularities $H(d\alpha)$ of the 2-form $d\alpha$ play a special role: they are the \textit{abnormal geodesics} of the corresponding Pontryagin maximum principle and they can be obtained formally by the Euler-Lagrange equations (\ref{el}) for $m=0$.

For the case $\nu \neq 0$, the geometry of this variational problem for the values $m\neq 0$ can be studied in terms of Jacobi metrics and for most of the cases where the 1-form $L_1=\alpha$ of gyroscopic forces is nonsingular, in the sense that the  2-form $\omega=da$ is nondegenerate (symplectic) on $M$.   It is important to notice also that for any $m\neq 0$ the Legendre transform $\mathcal{L}_L:TM\rightarrow T^*M$ of $L$ is a diffeomorphism and thus the phase space of the Lagrangian system can be identified with cotangent bundle $T^*M$ with Hamiltonian $F=(\mathcal{L}^{-1}_L)^*L$ and symplectic form $\Omega=dp\wedge dx$ the natural symplectic form of the cotangent bundle. The generalised Euler-Lagrange equations (\ref{el}) transform in that way to the canonical Hamilton's equations:
\[X_F\lrcorner \Omega=dF.\]

We will be interested here in the \textit{quantisation limit equations}, i.e. for $m\rightarrow 0$ (or $e\rightarrow \infty$, $\nu \rightarrow \infty$). Notice that for $m=0$ the Lagrangian $L=eL_1+\nu L_0$ is linear in the velocities and thus its Legendre tranform $\mathcal{L}_L:TM\rightarrow T^*M$ is not a diffeomorphism. For this reason, Lagrangians linear in the velocities are called \textit{singular (or constrained)} and the dynamics that they define through the Euler-Lagrange equations:
\begin{equation}
\label{els}
e\dot{x}\lrcorner da=\nu df,
\end{equation}
is also called \textit{singular (or constrained) Lagrangian dynamics}. The equation (\ref{els}) above can be viewed as a Constrained Hamiltonian System $(f,\omega)$ with 2-form $\omega=da$ the form of gyroscopic forces and ``Hamiltonian'' $f$ defined by the potential energy $L_0$ of the initial system. To describe this geometrically notice that in the limit $m=0$, the image of the Legendre tranform $\mathcal{L}_L(TM)\subset T^*M$ defines naturally a constraint submanifold in the phase space which is exactly equal to the image of the 1-form $L_1=\alpha$, viewed as a local section $\alpha:M\rightarrow T^*M$ of the cotangent bundle:
\[Im\mathcal{L}_L=Im\alpha \subset T^*M\]
and it is thus diffeomorphic to the configuration space $M$. One has thus defined a diagram of maps:
\begin{equation}
\label{d1}
M\stackrel{\alpha}{\hookrightarrow} T^*M\stackrel{F}{\rightarrow} \mathbb{R}, 
\end{equation}   
whose left arrow is the embedding of $M$ in $T^*M$ through $\alpha$ and the right arrow is the Hamiltonian $F$. This shows that the motions of the generalised particle can be viewed inside the general scheme of the theory of Hamiltonian systems with constraints. As is easily verified, the  Euler-Lagrange equations (\ref{el}) for $m=0$, correspond to the restriction of the Hamiltonian system $(F,\Omega)$ on the constrained submanifold, i.e. on the image of the 1-form $\alpha$. Indeed, one has
\[\alpha^*\Omega=da, \quad \alpha^*F=f,\]
and thus the restriction $\alpha^*(F,\Omega)$ defines the Constrained Hamiltonian System:
\[X_f\lrcorner da=df,\]
which is exactly the system of Euler-Lagrange equations (\ref{els}) (were we have put $e=\nu=1$).

Now let as consider the problem to determine the motions of the particle in the quantisation limit. Fix a point $x_0 \in M$ and identify the germs of singular Lagrangians $L=L_1+L_0$ at $x_0$ with the germs of the corresponding pairs of potentials $L:=(\alpha,f)$. A singular Lagrangian $L$ will be called ``generic'' if the corresponding pair of potentials $(\alpha, f)$ is in general position (relative to diagram (\ref{d1})) or equivalently, the codimension of the singularities of the pair $(d\alpha,f)$ is less or equal to $2=dimM$. The fact that this definition is correct is veryfied by the Darboux-Givental theorem \cite{Gi}.  

Replace now the 1-form $\alpha$ by a 1-form $\alpha+d\xi$, where $\xi$ is some arbitrary function, and the potential $f$ by $f+c$, $c$ some constant. Then the form of the Euler-Lagrange equations (\ref{els}) does not change and so there is naturally defined an equivalence relation between singular Lagrangians: 
\begin{definition}
Two germs $L=(\alpha,f)$ and $L'=(\alpha',f')$ of singular Lagrangians at $x_0 \in M$ will be called variationally (or gauge) equivalent if their Euler-Lagrange equations (\ref{els}) are equivalent, i.e. there exists a diffeomorphism germ $\Phi$, $\Phi(x_0)=x_0$, an arbitrary function germ $\xi$ and a constant $c$ such that:
\[\Phi^*\alpha'=\alpha+d\xi, \quad \Phi^*f'=f+c.\] 
\end{definition}    

From the results of the previous sections on the classification of the pair $(d\alpha,f)$ we immediately obtain:
\begin{theorem}
\label{tl}
The germ of a generic analytic Lagrangian $L$ at a point $x_0$ on a manifold $M$, is variationally equivalent to one of the following four invariant normal forms:
\begin{equation}
\label{lnf0} 
L(x,\dot{x})=x_1\dot{x}_2-x_1,
\end{equation}
\begin{equation}
\label{lnf1} 
L(x,\dot{x})=x_1\dot{x}_2-\phi(\pm x_1^2\pm x_2^2),
\end{equation}
\begin{equation}
\label{lnf2} 
L(x,\dot{x})=\frac{x_1^2}{2}\dot{x_2}\mp x_2, 
\end{equation}
\begin{equation}
\label{lnf3} 
L(x,\dot{x})=\frac{x_1^2}{2}\dot{x}_2-\psi(x_1\pm x_2^2), 
\end{equation}
where the functions germs $\psi$ and $\phi$ are analytic functions of one variable with a simple zero at the origin and they are the unique functional invariants of the (variational orbits of the) corresponding Lagrangians. 
\end{theorem}
\begin{proof}
Normal forms (\ref{lnf0}), (\ref{lnf1}) correspond to the well known regular and Morse cases respectively at a point $x_0\in M\setminus H(\omega)$ in the symplectic plane. Local normal forms (\ref{lnf2}) and $(\ref{lnf3})$ correspond to regular, transversal points $x_0\in H(\omega)=\{x_1=0\}$ of $f$ on the Martinet curve (open and dense) and to (isolated) points of first order tangency of $f$ with the Martinet curve respectively. Normal form (\ref{lnf3}) is obtained immediately from Corollary \ref{t3} of Theorem \ref{t2}. The normal form (\ref{lnf2}) is a simple exercise; the $\mp$ sign comes from the fact that in the real case the Martinet curve $H(\omega)=\{x_1=0\}$ has an invariant orientation induced by the two symplectic structures in its complement in $M$. In fact, there is no real analytic diffeomorphism preserving $x_1dx_1\wedge dx_2$ and sending $x_2$ to $-x_2$.
\end{proof}

%


\section*{Acknowledgements}
The author whould like to thank Jeroen Lamb, Dmitry Turaev as well as Michail Zhitomirskii and Jean-Pierre Francoise for useful discussions and their attention on the work.



\end{document}